\documentclass[a4paper]{article}
\usepackage{fullpage}
\usepackage{tikz}

\usepackage[pdftex,breaklinks,colorlinks,
linkcolor=blue,
citecolor=blue,
urlcolor=blue]{hyperref}

\usepackage{enumitem}
\usepackage{amsmath,amssymb,amsfonts,amsthm}

\newtheorem{theorem}{Theorem}

\newtheorem{lemma}[theorem]{Lemma}
\newtheorem{proposition}[theorem]{Proposition}

\theoremstyle{definition}
\newtheorem{definition}[theorem]{Definition}
\newtheorem{remark}[theorem]{Remark}

\author{Pablo Romero\footnote{Facultad de Ingenier\'ia, Universidad de la Rep\'ublica, Montevideo, Uruguay. E-mail address: \texttt{promero@fing.edu.uy}}
\footnote{Facultad de Ciencias Exactas y Naturales, Universidad de Buenos Aires, Ciudad Universitaria. Av. Int. Güiraldes 2160. Buenos Aires, Argentina.}}

\date{}

\makeatletter%
\begin{document}

\title{Characterization of locally most split reliable graphs}

\maketitle

\begin{abstract}\let\thefootnote\relax
A two-terminal graph is a graph equipped with two distinguished vertices, called terminals.  Let $T_{n,m}$ be the set of all nonisomorphic connected simple two-terminal graphs on $n$ vertices and $m$ edges. Let $G$ be any two-terminal graph in $T_{n,m}$. For every number $p$ in $[0,1]$ we let each of the edges in $G$ be independently deleted with probability $1-p$. The \emph{split reliability} $SR_{G}(p)$ is the probability that the resulting spanning subgraph has precisely $2$ connected components, each one including one terminal. The two-terminal graph $G$ is \emph{uniformly most split reliable} if $SR_G(p)\geq SR_{H}(p)$ for each $H$ in $T_{n,m}$ and every $p$ in $[0,1]$. 
We say $G$ is \emph{locally most split reliable} if there exists $\delta>0$ such that 
$SR_G(p)\geq SR_{H}(p)$ for each $H$ in $T_{n,m}$ and every $p$ in $(1-\delta,1)$. 

Brown and McMullin showed that there exists uniformly most split reliable graphs in each class $T_{n,m}$ such that $m=n-1$, $m=\binom{n}{2}$, or $m=\binom{n}{2}-1$. The authors also proved that there is no uniformly most split reliable two-terminal graph in $T_{n,n}$ when $n\geq 6$ and 
specified in which classes $T_{n,m}$ such that $n\leq 7$ there exist uniformly most split reliable graphs. The existence or nonexistence of uniformly most split reliable graphs in the remaining cases is posed by Brown and McMullin as an open problem. 

In this work, the set $\mathcal{G}_{n,m}$ consisting of all locally most split reliable graphs is characterized in each nonempty class $T_{n,m}$. It is proved that a graph in $T_{n,m}$ is locally most split reliable if and only if 
its split reliability equals that of the balloon graph equipped with two terminals whose distance equals its diameter. Finally, it is proved that there is no  uniformly most split reliable graph in $T_{n,m}$ when $n\geq 7$ and $n\leq m \leq \binom{n-3}{2}+3$.
\end{abstract}

\renewcommand{\labelitemi}{--}

\section{Introduction}\label{section:motivation}
In a seminal article~\cite{1956-MooreShannon}, Moore and Shannon proved using probabilistic methods that highly reliable circuits can be constructed from a sufficiently large number of identical imperfect relays. The key in their construction is to maximize the connectedness probability between two fixed points in a circuit using imperfect components, what is now known as the two-terminal reliability. It is known~\cite{1983-ProvanBall} that the two-terminal reliability evaluation in general graphs belongs to the class of NP-Hard problems. From an optimization viewpoint, if we are given a fixed number of points (vertices) including two terminals and a fixed number of imperfect components (edges) with identical failure probability, a question that arises is how to connect these points using such components in order to maximize the two-terminal reliability, regardless of the failure probability. This problem has been mathematically formalized by Bertrand et al.~\cite{Bertrand-2018}. Remarkably, it was recently proved that, except for very few graph classes, there is no uniformly most reliable two-terminal graphs~\cite{Xie-2021}. The locally most reliable two-terminal graphs for large failure probabilities were determined in~\cite{Gong-2024}, except for very few graph classes as well. The interested reader can consult the survey~\cite{2022-Romero-Survey} for  details.
 
In this work we will study an alternative problem that was recently proposed by Brown and McMullin~\cite{2023-Brown}. We are given a fixed number of vertices including two terminals and a fixed number of edges each one with identical failure probability. We want to find a two-terminal graph that maximizes the probability that the graph is split into precisely two parts each one including one terminal, regardless of the failure probability of the edges. If such two-terminal graph exists, it is called uniformly most split reliable graph. 
The practical interest in the determination of uniformly most split reliable graphs appears in the dissemination of nonconflicting messages transmitted by $2$ different sources.
In an analogy with~\cite{Gong-2024}, a characterization of the set $\mathcal{G}_{n,m}$ consisting of all the locally most split reliable graphs is given in each of the nonempty classes $T_{n,m}$ of connected simple two-terminal graphs on $n$ vertices and $m$ edges. In an analogy with the previous work~\cite{Xie-2021}, here it is proved that, in most graph classes, there is no uniformly most split reliable two-terminal graph. A description of the problem as well as its origin is given in the following paragraphs.\\

Let $G$ be any connected multigraph. The \emph{reliability polynomial} $R_G(p)$ is the probability of $G$ being connected after each of its edges is independently deleted with probability $1-p$. 
Brown and Colbourn~\cite{1992-BrownColbourn} conjectured that the roots of the reliability polynomial always lie in the closed unit disc of the complex plane. Royle and Sokal~\cite{2004-Royle} proved the conjecture false, finding roots of modulus larger than $1$ by a slim margin. Later, Brown and Mol~\cite{2017-Brown} constructed graphs whose reliability roots have larger modulus than any previously known. An essential tool in their construction is the introduction of the \emph{split reliability}. A \emph{two-terminal multigraph} is a multigraph equipped with two distinguished vertices, called terminals.  
Let $M_{n,m}$ be the set consisting of all two-terminal multigraphs on $n$ vertices and $m$ edges. 
Let $p\in [0,1]$ and $G$ in $M_{n,m}$. Let each of the edges in $G$ be independently deleted with probability $1-p$. 
The \emph{split reliability of $G$}, denoted $SR_G(p)$, is the probability that the resulting spanning subgraph of $G$ has precisely two connected components each one including one of its terminals. A graph $G$ is \emph{optimal} if $SR_G(p)\geq SR_H(p)$ for each $H$ in $M_{n,m}$ and every $p$ in $[0,1]$.  
Brown and McMullin~\cite{2023-Brown} succeeded to determine whether or not there exists some optimal graph in each nonempty class $M_{n,m}$. In their conclusions the authors give some comments on the corresponding problem of finding \emph{optimal simple graphs} in the restricted classes of connected \emph{simple} two-terminal graphs $T_{n,m}$ (i.e., whether or not there exists some $G$ in $T_{n,m}$ such that $SR_G(p)\geq SR_H(p)$ for each $H$ in $T_{n,m}$ and every $p$ in $[0,1]$). On the one hand, using 
a simple argument they proved that there is no optimal simple graphs in each class $T_{n,m}$ such that $m=n$ and $n\geq 6$. On the other hand, they proved that there exist optimal simple graphs in each nonempty class $T_{n,m}$ such that $m=n-1$, $m=\binom{n}{2}$, or $m=\binom{n}{2}-1$. 
Additionally, calculations on small graphs show that, among all classes $T_{n,m}$ such that $n\leq 7$, there exist optimal simple graphs if and only if one of the following conditions holds: (i) $n\leq 5$, (ii) $n=6$ and $m\notin \{6,8\}$, or (iii) $n=7$ and $m\in \{6,14,15,\ldots,21\}$. Finally, they stated that ``the existence or nonexistence of optimal simple graphs in other cases is left as an open problem worth exploring.'' Thus far, the determination of optimal simple graphs was only achieved for the aforementioned classes $T_{n,m}$.\\ 

Here we adapt the terminology presented in~\cite{2023-Brown} and use the term \emph{uniformly most split reliable graph} instead of optimal simple graph for descriptive reasons. The main results of this work are Theorem~\ref{theorem:local1} and Theorem~\ref{theorem:nonexistence}. On the one hand, Theorem~\ref{theorem:local1} gives a characterization of the locally most split reliable graphs in each nonempty class $T_{n,m}$. 
On the other hand, Theorem~\ref{theorem:nonexistence} asserts that there is no uniformly most split reliable graph in $T_{n,m}$ when $n\geq 7$ and $n\leq m \leq \binom{n-3}{2}+3$. 

The document is organized as follows. Section~\ref{section:strategy} presents some concepts related with split reliability, the statements of Theorem~\ref{theorem:local1} and Theorem~\ref{theorem:nonexistence}, 
and a proof strategy for both theorems. As it will turn out, the most split reliable graphs consists precisely of the balloon graph~\cite{1996-Petingi} equipped with two terminals whose distance equals the diameter of the balloon together with all two-terminal graphs whose split reliability equals that of the balloon graph. Consequently, Section~\ref{section:properties} presents some remarkable properties satisfied by balloon graphs. 
A characterization of locally most split reliable two-terminal graphs is given in Section~\ref{section:class}. Finally, in Section~\ref{section:main} it is proved that there is no uniformly most split reliable graph in $T_{n,m}$ when $n\geq 7$ and $n\leq m \leq \binom{n-3}{2}+3$. 

\section{Proof strategy}\label{section:strategy}
In this section we will first give some key concepts related with split reliability in simple graphs. Then, we will close this section with the proof strategy of the main results of this work, namely, Theorem~\ref{theorem:local1} and Theorem~\ref{theorem:nonexistence} which state, respectively, the existence and construction of locally most split reliable graphs in each nonempty class $T_{n,m}$, 
and the nonexistence of uniformly most split reliable graphs in $T_{n,m}$ when $n\geq 7$ and $n\leq m \leq \binom{n-3}{2}+3$.\\

A pair of two-terminal graphs $G$ and $H$ are \emph{isomorphic}, and it will be denoted $G \cong H$,
if there exists some graph isomorphism $\varphi:V(G) \to V(H)$ such that the terminal set $\{s,t\}$ is invariant (i.e., either $\varphi(s)=s$ and $\varphi(t)=t$, or $\varphi(s)=t$ and $\varphi(t)=s$). 
Let $T_{n,m}$ be the set of all nonisomorphic connected simple two-terminal graphs on $n$ vertices and $m$ edges. Throughout this article we will use the symbols $s$ and $t$ for the terminal vertices unless specified otherwise. Let $G$ be any graph in $T_{n,m}$. A \emph{split subgraph of $G$} is a spanning subgraph of $G$ having precisely $2$ connected components, each one including  one terminal. For each $p$ in $[0,1]$, the \emph{split reliability of $G$ evaluated at $p$}, denoted $SR_G(p)$, is the probability of having a split subgraph of $G$ after each of its edges is independently deleted with probability $1-p$. 
For each $i\in \{0,1,\ldots,m\}$ we define $N_i(G)$ as the number of split subgraphs of $G$ having precisely $i$ edges. Clearly,
\begin{equation*}
SR_{G}(p) = \sum_{i=n-2}^{m-1}N_i(G)p^i(1-p)^{m-i}.    
\end{equation*}

\begin{definition} A two-terminal graph $G$ in $T_{n,m}$ 
is \emph{uniformly most split reliable} if for each $H$ in $T_{n,m}$ and every $p$ in $[0,1]$, $SR_{G}(p)\geq SR_{H}(p)$.
\end{definition}

\begin{definition} A two-terminal graph $G$ in $T_{n,m}$ 
is \emph{locally most split reliable} if for each $H$ in $T_{n,m}$ 
there exists $\delta>0$ such that $SR_{G}(p)\geq SR_{H}(p)$ for every $p$ in $(1-\delta,1)$. 
\end{definition}

For each $G$ in $T_{n,m}$ and each $i$ in $\{0,1,\ldots,m\}$ we define the number $F_i(G)$ by $F_i(G)=N_{m-i}(G)$. Lemma~\ref{lemma:local} will be essential in the proof of Theorem~\ref{theorem:local1}. 
It can be proved using elementary analysis.
\begin{lemma}[Brown and McMullin~\cite{2023-Brown}] \label{lemma:local}
Let $G$ and $H$ be in $T_{n,m}$. The following assertions hold:
\begin{enumerate}[label=(\roman*)]
\item\label{l1} If there exists $i\in \{0,1,\ldots,m\}$ such that $N_k(G)=N_k(H)$ for all $k\in \{0,1,\ldots,i-1\}$ and $N_i(G)>N_i(H)$, then there exists $\delta>0$ such that $SR_{G}(p)>SR_H(p)$ for all $p\in (0,\delta)$.
\item\label{l2} If there exists $j\in \{0,1,\ldots,m\}$ such that $F_k(G)=F_k(H)$ for all $k\in \{0,1,\ldots,j-1\}$ and $F_j(G)>F_j(H)$, then there exists $\delta>0$ such that $SR_{G}(p)>SR_H(p)$ for all $p\in (1-\delta,1)$.
\end{enumerate}
\end{lemma}

Now, we will propose a method to construct the set $\mathcal{G}_{n,m}$ consisting of locally most split reliable graphs in each nonempty set $T_{n,m}$. The following concept will be essential.

\begin{definition}
A pair of two-terminal graphs $G$ and $H$ in $T_{n,m}$ are \emph{split-equivalent} if $SR_G(p)=SR_H(p)$ for every $p$ in $[0,1]$.     
\end{definition}

Define, for each graph $G$ in $T_{n,m}$, the tuple $F(G)$ given by $(F_1(G),F_2(G),\ldots,F_m(G))$. 
Recall that the set consisting of Bernstein polynomials $\{b_{i,m}: i\in \{0,1,\ldots,m\}\}$ where $b_{i,m}(x)=\binom{m}{i}x^i(1-x)^{m-i}$ is a basis of the subspace of polynomials with degree not greater than $m$. As a consequence, the following remark holds.
\begin{remark}\label{remark:splitequivalent}
A pair of two-terminal graphs $G$ and $H$ in $T_{n,m}$ are split-equivalent if and only if 
$F(G)=F(H)$.
\end{remark}
Let $T_{n,m}$ be any nonempty set of two-terminal graphs. Clearly, $T_{n,m}$ can be partitioned 
into finitely many classes consisting of split-equivalent two-terminal graphs. 
For each $G$ in $T_{n,m}$, let $[G]$ be the class of two-terminal graphs that are equivalent to $G$, i.e., $[G]=\{H: H\in T_{n,m}, \, F(G)=F(H)\}$.

Now, let $\mathcal{G}_{n,m}$ the set consisting of locally most split reliable graphs in $T_{n,m}$. 
Observe that $\mathcal{G}_{n,m}$ equals $[G]$ for some two-terminal graph $G$ in $T_{n,m}$. In fact, if we pick any pair of elements $G_1$ and $G_2$ in $\mathcal{G}_{n,m}$ 
then their split reliability polynomials coincide in an interval of $[0,1]$ with positive measure thus $SR_{G_1}(p)=SR_{G_2}(p)$ and $[G_1]=[G_2]$. 
From Lemma~\ref{lemma:local}\ref{l2}, each locally most split reliable graph $G$ in $\mathcal{G}_{n,m}$ maximizes the tuple $F(G)$ using the lexicographic order among all two-terminal graphs $G$ in $T_{n,m}$. 
The determination of the tuple $F(G)$ for an arbitrary two-terminal graph $G$ in $T_{n,m}$ belongs to the class of $\#$P-Complete problems~\cite{2023-Brown}. 
Nevertheless, in the following paragraphs, we will propose a method to find $\mathcal{G}_{n,m}$. This method will be accomplished in 
Section~\ref{section:class}.
 
Define $T_{n,m}^{(0)}$ as $T_{n,m}$ and, for each $i\in \{0,1,\ldots,m-1\}$ we define $T_{n,m}^{(i+1)}$ as $\{G:G\in T_{n,m}^{(i)}\, \, F_{i+1}(G)\geq F_{i+1}(H)\,\text{ for each } H \text{ in } T_{n,m}^{(i)}\}$. A consequence of Lemma~\ref{lemma:local} is the following result.
\begin{lemma}\label{lemma:construction}
The set $\mathcal{G}_{n,m}$ consisting of all most split reliable graphs in each nonempty class $T_{n,m}$ is precisely the nonempty set $T_{n,m}^{(m)}$, i.e., $\mathcal{G}_{n,m}=T_{n,m}^{(m)}$.
\end{lemma}
\begin{proof}
As the set $T_{n,m}$ is both finite nonempty then a finite induction shows that, for each $i\in \{0,1,\ldots,m\}$ 
it holds that $T_{n,m}^{(i)}$ is also finite and nonempty. In particular, the set $T_{n,m}^{(m)}$ is finite and nonempty. By construction, each pair of two-terminal graphs $G_1$ and $G_2$ in $T_{n,m}^{(m)}$ satisfy that $F(G_1)=F(G_2)$. 
Remark~\ref{remark:splitequivalent} gives that $G_1$ and $G_2$ are 
split-equivalent. Now, let $G$ be any element in $T_{n,m}^{(m)}$ and $H$ be any element in $T_{n,m}-T_{n,m}^{(m)}$. 
As $H$ is not in $T_{n,m}^{(m)}$, we can pick the first index $j$ in $\{1,2,\ldots,m\}$ such that $H$ is not in $T_{n,m}^{(j)}$. As $G$ is in $T_{n,m}^{(j)}$ we know that $F_{j}(G)>F_j(H)$ and further, $F_k(G)=F_k(H)$ for all $k\in \{0,1,\ldots,j-1\}$. Lemma~\ref{lemma:local}\ref{l2} gives that there exists $\delta>0$ such that $SR_{G}(p)>SR_H(p)$ for all $p\in (1-\delta,1)$. 
The lemma follows.
\end{proof}

\begin{remark}\label{remark:stop}
If there exists $i$ in $\{1,2,\ldots,m\}$ such that each pair of two-terminal graphs in $T_{n,m}^{(i)}$ are split-equivalent, then $T_{n,m}^{(i)}=\mathcal{G}_{n,m}$. 
\end{remark}

A key to characterize $\mathcal{G}_{n,m}$ is the concept of balloon graph given in Definition~\ref{def:balloon}. 
Let $I=\{(n,m)\in \mathbb{Z}^2: n\geq 4, \, n \leq m \leq \binom{n}{2}\}$. 
Observe that the set $I$ consists of all pairs $(n,m)$ such that $T_{n,m}$ is nonempty, with the exception of the pair $(3,3)$ and all pairs $(n,m)$ such that $m=n-1$ and $n\geq 2$, which 
were already covered by Brown and McMullin~\cite{2023-Brown}. In fact, it is clear that $T_{3,3}$ has a single element which is $K_3$ equipped with two-terminals thus $\mathcal{G}_{3,3}=\{K_{3,3}\}$, and the authors in~\cite{2023-Brown} proved that, for each $n$ such that $n\geq 2$,  the $n$-path $P_n$ equipped with its endpoints as terminals is the only uniformly split reliable graph in $T_{n,n-1}$ thus $T_{n,n-1}=\{P_n\}$. In the following we will study, without loss of generality, the classes $T_{n,m}$ such that $(n,m)$ belongs to the set $I$. 

Consider the partition $\{I_0,I_1\}$ of $I$  such that $I_1=I-I_0$ and $I_0$ is defined as follows,
\begin{equation*}
I_0 = \left\{(n,m)\in \mathbb{Z}^2: n\geq 4, \, \, \binom{n-1}{2}+2 \leq m \leq \binom{n}{2}\right\}.    
\end{equation*}

\begin{definition}\label{def:balloon}
For each $(n,m)$ in $I$ \emph{the balloon graph} $B_{n,m}$ in $\mathcal{C}_{n,m}$ is defined recursively as follows. 
\begin{itemize}
    \item If $(n,m) \in I_0$ then $B_{n,m}$ consists of $K_{n-1}$ plus a vertex $v$ connected to $m-\binom{n-1}{2}$ vertices in $K_{n-1}$. 
    \item If $(n,m)=(4,4)$ then $B_{n,m}$ consists of $K_3$ plus a vertex hanging to it. 
    \item Otherwise, $B_{n,m}$ is $B_{n-1,m-1}$ plus a vertex $v$ hanging to some minimum degree vertex in $B_{n-1,m-1}$.
\end{itemize}
\end{definition}
The graph $B_{9,15}$ is depicted in Figure~\ref{fig:GyH}. Balloon graphs were introduced in~\cite{1996-Petingi} for pairs $(n,m)$ in $I_0$. Definition~\ref{def:balloon} extends this concept for all pairs in $I$. 

Let $G$ be any connected simple graph. For each pair of vertices $u$ and $v$ in $G$, we denote $d_G(u,v)$ 
the distance between $u$ and $v$ in $G$. The diameter of $G$ is the maximum distance among all pairs of vertices in $G$. 

\begin{definition}
For each $(n,m)$ in $I$, the \emph{two-terminal balloon graph} $G_{n,m}$ is the balloon graph $B_{n,m}$ equipped with two terminals $s$ and $t$ whose distance equals the diameter of $B_{n,m}$. 
\end{definition}

It is simple to check that the two-terminal balloon graph $G_{n,m}$ is unique up to isomorphism. 

Theorem~\ref{theorem:local1} gives a characterization of locally most split reliable graphs.

\begin{theorem}\label{theorem:local1}
For each pair of integers $(n,m)$ in $I$, the locally most split reliable graphs $\mathcal{G}_{n,m}$ in $T_{n,m}$ is the set consisting of all two-terminal graphs in $T_{n,m}$ that are equivalent to the two-terminal balloon graph $G_{n,m}$, i.e.,
\begin{equation*}
\mathcal{G}_{n,m}=[G_{n,m}]=\{H:H\in T_{n,m}, \, \, F(H)=F(G_{n,m})\}.    
\end{equation*}
\end{theorem}

In Section~\ref{section:main} we will construct two-terminal graphs $H_{n,m}$ satisfying the conditions of Proposition~\ref{proposition:near0}.
\begin{proposition}\label{proposition:near0}
For each pair of integers $n$ and $m$ such that $n\geq 7$ and $n\leq m \leq \binom{n-3}{2}+3$ there exists a two-terminal graph $H_{n,m}$ in $T_{n,m}$ such that $N_{n-2}(H_{n,m})>N_{n-2}(G_{n,m})$.
\end{proposition}

A consequence of Theorem~\ref{theorem:local1} and Proposition~\ref{proposition:near0} is the following theorem.
\begin{theorem}\label{theorem:nonexistence}
For each pair of integers $n$ and $m$ such that $n\geq 7$ and $n\leq m \leq \binom{n-3}{2}+3$ there is no uniformly most split reliable graph in $T_{n,m}$.    
\end{theorem}
\begin{proof}
Assume by contradiction that $G^*$ is uniformly most split reliable in $T_{n,m}$. As $\mathcal{G}_{n,m}$ is the set of locally most split reliable graphs in $T_{n,m}$, $G^*$ belongs to $\mathcal{G}_{n,m}$. By Theorem~\ref{theorem:local1}, 
we know that $G^*$ and $G_{n,m}$ are split equivalent thus $F(G^*)=F(G_{n,m})$. 
By Proposition~\ref{proposition:near0}, there exists $H_{n,m}$ in $T_{n,m}$ such that 
$N_{n-2}(H_{n,m})>N_{n-2}(G_{n,m})$. Consequently, $N_{n-2}(H_{n,m})>N_{n-2}(G^*)$. Applying Lemma~\ref{lemma:local}\ref{l1} using $i=n-2$ we obtain that there exists $\delta>0$ such that $SR_{H_{n,m}}(p)>SR_{G^*}(p)$ for all $p\in (0,\delta)$, which contradicts that $G^*$ is uniformly most split reliable. Then, there is no uniformly most split reliable graph in $T_{n,m}$ and the theorem follows. 
\end{proof}

In Section~\ref{section:properties} we will present remarkable properties of balloon graphs that will be used in the proofs of Theorem~\ref{theorem:local1} and Proposition~\ref{proposition:near0}. Then, a proof of Theorem~\ref{theorem:local1} is presented in Section~\ref{section:class}. Finally, a proof of Proposition~\ref{proposition:near0} is presented in Section~\ref{section:main}. 

\section{Properties of balloon graphs}\label{section:properties}
In this section we will characterize all graphs maximizing the number of bridges, minimizing the edge connectivity, and minimizing the number of spanning trees, among all connected simple graphs on $n$ vertices and $m$ edges. Surprisingly, the balloon graph is a simultaneous solution to the three combinatorial optimization problems. This fact is essential to prove Theorem~\ref{theorem:local1}. Finally, we include some results credited to Bogdanowicz~\cite{2009-Bogdanowicz} related to the minimization of spanning trees among all connected graphs on $n$ vertices and $m$ edges which will be useful to prove Proposition~\ref{proposition:near0}.  

Some concepts on graph theory are revisited in Subsection~\ref{subsection:concepts}. 
Subsection~\ref{subsection:3cops} presents the previous combinatorial optimization problems together with their corresponding solutions. A new characterization of graphs having the maximum number of bridges is also included.

\subsection{Concepts}\label{subsection:concepts}
Let $\mathcal{C}_{n,m}$ be the set of connected simple graphs on $n$ vertices and $m$ edges. Let $G$ be any graph in $\mathcal{C}_{n,m}$. The \emph{minimum degree of $G$} is denoted $\delta(G)$.   
A \emph{bridge in $G$} is an edge $e$ in $G$ such that $G-e$ is not connected. The number of bridges in $G$ is denoted $b(G)$. 
The \emph{subdivision of an edge $e$ in $G$} is the graph obtained by replacing the edge $e$ with endpoints $v$ and $w$ by two edges $vz$ and $zw$  where $z\notin V(G)$. The \emph{contraction of an edge $e$ in $G$} is a new graph, denoted $G*e$, that is obtained by the deletion of the edge $e$ followed by the identification of the endpoints of $e$. Observe that, if $e$ is a bridge, then $G*e \in \mathcal{C}_{n-1,m-1}$. The \emph{skeleton} of $G$, here denoted $G'$, is the graph that arises from $G$ by the contraction of each of its bridges. 

The $n$-path, the $n$-cycle, and the $n$-complete graph are denoted $P_n$, $C_n$, and $K_n$, respectively. If $G$ is not $K_1$ then its \emph{edge connectivity}, denoted  $\lambda(G)$, is the minimum number of edges that must be removed from $G$ to obtain a subgraph that is not connected. 
A \emph{minimum separator} of $G$ is an edge set $U$ of $G$ having precisely $\lambda(G)$ edges 
such that $G-U$ is not connected. The number of minimum separators of $G$ is denoted $s(G)$. 
A \emph{tree} is a connected simple graph without cycles. The number of spanning trees of a simple graph $G$ is denoted $t(G)$.

\subsection{Three combinatorial optimization problems}\label{subsection:3cops}
Let us consider, for each pair of integers $(n,m)$ in $I$, the following combinatorial optimization problems,
\begin{align*}
b(n,m) &= \max\{b(G): G\in \mathcal{C}_{n,m}\},\\
\lambda(n,m) &= \min\{\lambda(G): G\in \mathcal{C}_{n,m}\},\\
t(n,m) &= \min\{t(G): G\in \mathcal{C}_{n,m}\}.
\end{align*}

Surprisingly, we will see that for each pair $(n,m)$ in $I$ the balloon graph $B_{n,m}$ is a simultaneous solution to the three problems, i.e., 
$b(B_{n,m})=b(n,m)$, $\lambda(B_{n,m})=\lambda(n,m)$, and $t(B_{n,m})=t(n,m)$. 
In~\cite{2022-Romero} it is proved that for each $(n,m)$ in $I$ it holds that  $b(n,m)=n-1-\lceil \sqrt{2(m-n+3)}-\frac{1}{2} \rceil$ and further, that $B_{n,m}$ has precisely $b(n,m)$ bridges.
\begin{proposition}[Romero~\cite{2022-Romero}]\label{proposition:balloon}
For each $(n,m)$ in $I$, $b(B_{n,m})=b(n,m)$.    
\end{proposition}

Here we characterize all graphs in $\mathcal{C}_{n,m}$ having precisely $b(n,m)$ bridges in terms of its skeleton. Let $\mathcal{B}_0$ be the set of all connected simple graphs $G$ in $\mathcal{C}_{n,m}$ for some pair $(n,m)$ in $I_0$, i.e.,
\begin{equation*}
\mathcal{B}_0 = \{G: G\in \mathcal{C}_{n,m}, \, (n,m)\in I_0\}.    
\end{equation*}

\begin{proposition}\label{proposition:skeleton}
Let $(n,m)$ be in $I$. 
A graph $G$ in $\mathcal{C}_{n,m}$ has $b(n,m)$ bridges if and only if $G' \in \mathcal{B}_0$. 
\end{proposition}
\begin{proof}
Let $(n,m)$ be in $I$. First, consider a graph $G$ in $\mathcal{C}_{n,m}$ having precisely $b(n,m)$ bridges. 
Assume by contradiction that $G'\notin \mathcal{B}_0$. Then, $G'$ is a graph in $\mathcal{C}_{n',m'}$ for some pair $(n',m')$ that is not in $I_0$. As $G'$ is the skeleton of $G$ and $b(G)=b(n,m)$ we get that $m=m'+b(n,m)$.

As $(n',m')\notin I_0$, the balloon graph $B_{n',m'}$ has at least some hanging vertex thus 
$b(B_{n',m'})\geq 1$. Observe that $B_{n,m}$ is obtained from $B_{n',m'}$ by the addition of $m-m'$ bridges. By Proposition~\ref{proposition:balloon},  we get that $b(n,m)=b(B_{n,m})=m-m'+b(B_{n',m'})>m-m'=b(n,m)$ which is impossible thus $G' \in \mathcal{B}_0$. 

Finally, consider any graph $G$ in $\mathcal{C}_{n,m}$ such that $G' \in \mathcal{B}_0$. Let $n'$ and $m'$ be the number of vertices and edges in $G'$.  
By the definition of skeleton, $m=m'+b(G)$. 
Observe that we must contract precisely $m-m'$ bridges from $B_{n,m}$ to obtain the balloon graph $B_{n',m'}$ in $\mathcal{B}_0$ and $b(B_{n',m'})=0$. 
By Proposition~\ref{proposition:balloon} we get that 
$b(n,m)=b(B_{n,m})=m-m'=b(G)$, and the proposition follows.
\end{proof}

Now, let us determine $\lambda(n,m)$ for each pair $(n,m)$ in $I$.
The following technical lemma will be useful.
\begin{lemma}\label{lemma:technical}
Let $n$ be an integer such that $n\geq 4$. The function $f:\{2,3,\ldots,n-2\}\to \mathbb{R}$ defined by $f(a)=\binom{a}{2}+\binom{n-a}{2}$ satisfies that $f(a)\leq \binom{n-2}{2}+1$.    
\end{lemma}
\begin{proof}
It is clear that, for each $a\in \{2,3,\ldots,n-2\}$, $f(a)=f(n-a)$. Consequently, 
the maximum of $f$ is achieved in some integer $a$ such that $a\geq n/2$. Additionally, 
\begin{equation*}
f(a+1)-f(a)=\binom{a+1}{2}+\binom{n-a-1}{2}-\binom{a}{2}-\binom{n-a}{2} = 2a+1-n,    
\end{equation*}
thus $f$ is increasing when $a\geq n/2$ and $f(a)\leq f(n-2)=\binom{n-2}{2}+1$.
\end{proof}

Let us find $\lambda(G)$ for each $G$ in $\mathcal{B}_0$. 
\begin{lemma}\label{lemma:mindegree}
Each graph $G$ in $\mathcal{B}_0$ satisfies that $\lambda(G)=\delta(G)$. 
\end{lemma}
\begin{proof}
Let $G$ be any graph in $\mathcal{B}_0$. Clearly, $\lambda(G)\leq \delta(G)$. Assume by contradiction that $\lambda(G)<\delta(G)$. Let $U$ be some edge set in $G$ such that $|U|=\lambda(G)$ and 
$G-U=G_1\cup G_2$ for some simple graphs $G_1$ and $G_2$. 
Denote $n_1$ and $m_1$ (resp. $n_2$ and $m_2$) the number of vertices and edges of $G_1$ (resp. $G_2$). Clearly, $G-U$ has no isolated vertex thus $n_1\geq 2$ and $n_2\geq 2$. As $(n,m)\in I_0$ we know that $\binom{n-1}{2}+2\leq m$. Then, the following chain of inequalities holds,
\begin{equation*}
\binom{n-1}{2}+2\leq m = m_1+m_2+\lambda(G)\leq \binom{n_1}{2}+\binom{n-n_1}{2}+\lambda(G)\leq \binom{n-2}{2}+1+\lambda(G),    
\end{equation*}
where we used Lemma~\ref{lemma:technical} in the last inequality. This implies that
\begin{equation*}
\lambda(G)\geq \binom{n-1}{2}-\binom{n-2}{2}+1=n-1.    
\end{equation*}
As $\delta(G)>\lambda(G)$ and $\lambda(G)\geq n-1$ we would conclude that $\delta(G)>n-1$, which contradicts that $G$ is a simple graph. As a 
consequence, $\lambda(G)=\delta(G)$, and the lemma follows.
\end{proof}

We are in conditions to find a closed form for $\lambda(n,m)$.  
\begin{proposition}\label{proposition:lambda}
For each $(n,m)\in I$,  $\lambda(n,m)$ satisfies the following expression:
$$
\lambda(n,m)=
\begin{cases}
m-\binom{n-1}{2}, \text{ if } \, \, (n,m)\in I_0,\\
1, \text{ if } \, \, (n,m)\in I_1.
\end{cases}
$$
Additionally, if $(n,m)$ is any pair in $I_0$ then 
$B_{n,m}$ is the only graph in $\mathcal{C}_{n,m}$ such that $\lambda(B_{n,m})=\lambda(n,m)$. 
\end{proposition}
\begin{proof}
Let $(n,m)$ be any pair of integers in $I_0$. By Lemma~\ref{lemma:mindegree}, $\lambda(n,m)$ equals the minimum degree among all graphs in $\mathcal{C}_{n,m}$. Clearly, $B_{n,m}$ is the only graph in $\mathcal{C}_{n,m}$ with minimum degree in $\mathcal{C}_{n,m}$ 
thus $\lambda(n,m)=\lambda(B_{n,m})=m-\binom{n-1}{2}$. 
Finally, if $(n,m)\notin I_0$ then $b(B_{n,m})\geq 1$
and consequently $\lambda(n,m)=1$. The proposition follows.
\end{proof}

It is a simple exercise in combinatorics to prove that for each pair $(n,m)$ in $I_0$ 
there are precisely as many minimum separators in the balloon graph $B_{n,m}$ as vertices in 
$B_{n,m}$ with minimum degree. 
A consequence is the following remark.
\begin{remark}\label{remark:minimum}
If $(n,m)$ is in $I_0$ then the following assertions hold. 
\begin{enumerate}[label=(\roman*)]
\item If $m=\binom{n}{2}$ then $\lambda(B_{n,m})=n-1$. In this case $B_{n,m}$ is the $n$-complete graph $K_n$ and $s(B_{n,m})=n$.
\item If $m=\binom{n}{2}-1$ then $\lambda(B_{n,m})=n-2$. In this case $B_{n,m}$ is the $n$-complete graph minus an edge, and $s(B_{n,m})=2$.
\item If $m\leq \binom{n}{2}-2$ then $\lambda(B_{n,m})<n-2$. In this case $B_{n,m}$ has a single vertex with minimum degree and $s(B_{n,m})=1$.  
\end{enumerate}
\end{remark}

We close this section with two results authored by Bogdanowicz related to the determination of the number of spanning trees in balloon graphs and threshold graphs which will be essential for the proof of Theorem~\ref{theorem:nonexistence}. 

\begin{theorem}[Bogdanowicz~\cite{2009-Bogdanowicz}]\label{theorem:trees}
For each $(n,m)$ in $I$,  $t(B_{n,m})=t(n,m)$.    
\end{theorem}
Bogdanowicz also determined all graphs in $\mathcal{C}_{n,m}$ that attain the minimum number of spanning trees.

Recall that a \emph{threshold graph} $H=H(n;d_1,d_2,\ldots,d_k)$ consists of an $n-k$-complete graph with vertices $v_{k+1},\ldots,v_n$ plus an independent set on the remaining vertices $v_1,v_2,\ldots,v_{k}$ such that for each $i\in \{1,2,\ldots,k\}$ the vertex $v_i$ is adjacent to each of the vertices $v_{k+1},v_{k+2},\ldots,v_{k+d_i}$. 
Bogdanowicz also found, for each threshold graph $H$, the number of spanning trees $t(H)$. The following result will be useful for our purpose.
\begin{theorem}[Bogdanowicz~\cite{2009-Bogdanowicz}]
Suppose that the threshold graph $H=H(n;d_1,d_2,\ldots,d_k)$ is connected and $d_1\geq d_2\ldots \geq d_k$. Define $d_0=n-k$ and 
$d_{k+1}=1$. Then,
\begin{equation} \label{eq:bogdanowicz}
t(H(n;d_1,d_2,\ldots,d_k)) = (n-k)^{-2}\prod_{i=0}^{k}\left(d_i(n-k+i)^{d_i-d_{i+1}}\right).     
\end{equation}
\end{theorem}

Even though Bogdanowicz~\cite{2009-Bogdanowicz} successfully proved that each balloon graph $B_{n,m}$
attains the minimum number of spanning trees among all graphs 
in $\mathcal{C}_{n,m}$, it is not known whether or not $B_{n,m}$ is a uniformly least reliable graph in $\mathcal{C}_{n,m}$ (this conjecture was partially proved for all pairs $(n,m)$ in $I_0$; see~\cite{1996-Petingi}). An ambitious open problem is to determine whether or not each balloon graph $B_{n,m}$ 
is Tutte-minimum in $\mathcal{C}_{n,m}$ which, if affirmative, implies that $B_{n,m}$ is uniformly least reliable. 
The reader is invited to consult~\cite{Kahl-2022,Kahl-2023} for further details. 

\section{Locally most split reliable graphs}\label{section:class}
In this section we will find, for each pair $(n,m)$ in $I$,  the set $\mathcal{G}_{n,m}$ consisting of locally most split reliable graphs in $T_{n,m}$. 
Throughout this section we will assume that the pair $(n,m)$ belongs to the set $I$ and we define the numbers $n'$, $m'$ and $\lambda'$ by $n'=n-b(n,m)$, $m'=m-b(n,m)$, and $\lambda'=\lambda(n',m')$, respectively.\\ 

We will employ the methodology presented in Section~\ref{section:strategy}. Specifically, for each pair $(n,m)$ in $I$ we will successively find the sets $T_{n,m}^{(0)},T_{n,m}^{(1)},\ldots,T_{n,m}^{(i)}$ until we find the first index $i\in \{1,2,\ldots,m\}$ such that $T_{n,m}^{(i)}$ consists of split-equivalent two-terminal graphs. Finally, Remark~\ref{remark:stop} gives that $\mathcal{G}_{n,m}=T_{n,m}^{(i)}$.

To find $T_{n,m}^{(1)}$ we need to maximize $F_1(G)$ among all two-terminal graphs $G$ in $T_{n,m}$. As we will see in Lemma~\ref{lemma:starting}, the maximization of $F_1(G)$ among all two-terminal graphs $G$ in $T_{n,m}$ is straightforward. 
For each $G$ in $T_{n,m}$ we say that an edge $e$ is an \emph{$s$-$t$ cut} if $G-e$ has no path joining its terminals $s$ and $t$.  
\begin{lemma}\label{lemma:starting}
If $G$ is in $T_{n,m}$ then $F_1(G)\leq b(n,m)$, and the equality holds if and only if $G$ has $b(n,m)$ bridges that are $s$-$t$ cuts.
\end{lemma}
\begin{proof}
Let $G$ be any two-terminal graph in $T_{n,m}$. Clearly, $b(G)\leq b(n,m)$, and the equality holds if and only if $G$ has $b(n,m)$ bridges. As each $s$-$t$ cut in $G$ is also a bridge it holds that $F_1(G)\leq b(G)$, and the equality holds if and only if each bridge in $G$ is also an $s$-$t$ cut. 
Consequently, $F_1(G)\leq b(n,m)$, and the equality holds if an only if $G$ has precisely $b(n,m)$ bridges that are also $s$-$t$ cuts. 
The lemma follows. 
\end{proof}

Recall that $T_{n,m}^{(1)}$ consists of all two-terminal graphs $G$ in $T_{n,m}$ that maximize $F_{1}(G)$. 
Then, by Lemma~\ref{lemma:starting},
\begin{equation*}
T_{n,m}^{(1)} = \{G: G\in T_{n,m}, \, \, F_{1}(G)=b(n,m)\}.    
\end{equation*}
For each $G$ in $T_{n,m}^{(1)}$ we know that its skeleton $G'$ is a simple graph on $n'$ vertices and $m'$ edges. 
In Lemma~\ref{lemma:intermediate} we will study structural properties of the set $T_{n,m}^{(1)}$.

\begin{lemma}\label{lemma:intermediate}
Let $G$ be any graph in $T_{n,m}^{(1)}$. All the following assertions hold.
\begin{enumerate}[label=(\roman*)]
    \item\label{i1} The skeleton $G'$ belongs to $\mathcal{B}_0$.
    \item\label{i2} The edge connectivity of $G'$ is at least $\lambda'$, i.e., $\lambda(G')\geq \lambda'$.
    \item\label{i3} For each $i\in \{1,2,\ldots,\lambda'-1\}$, the number $F_i(G)$ does not depend on the choice of $G$.
\end{enumerate}
\end{lemma}
\begin{proof}
Let $G$ be any graph in $T_{n,m}^{(1)}$ we will prove each of the assertions separately in order.
\begin{enumerate}[label=(\roman*)]
\item As $F_1(G)=b(n,m)$, the graph $G$ has precisely $b(n,m)$ bridges and Proposition~\ref{proposition:skeleton} gives that its skeleton $G'$ is in $\mathcal{B}_0$, as required.
\item By assertion~\ref{i1} we know that $G'\in \mathcal{B}_0$. Consequently, Lemma~\ref{lemma:mindegree} gives that $\lambda(G')\geq \lambda'$.
\item Let $i$ be any integer in $\{1,2,\ldots,\lambda'-1\}$. 
By assertion~\ref{i2} we know that $\lambda(G')>i$. 
On the one hand, if we delete $0$ bridge of $G$ and $i$ edges of $G'$ then the resulting subgraph is connected thus it is not a split subgraph of $G$. On the other hand, if we delete $2$ or more bridges of $G$ then the resulting  subgraph has more than $2$ connected components thus it is neither a split subgraph of $G$. Consequently, the number of split subgraphs $F_i(G)$ is precisely the ways to choose $1$ out of the $b(n,m)$ bridges of $G$ plus $i-1$ out of the $m'$ edges of $G'$ thus $F_{i}(G)=b(n,m)\binom{m'}{i-1}$, and $F_i(G)$ does not depend on the choice of $G$. 
\end{enumerate}
\end{proof}

Lemma~\ref{lemma:intermediate}\ref{i3} gives that $T_{n,m}^{(\lambda'-1)}=T_{n,m}^{(1)}$.  
Therefore, our next step is to determine $T_{n,m}^{(\lambda')}$, or equivalently, 
all two-terminal graphs $G$ in $T_{n,m}^{(1)}$ that maximize $F_{\lambda'}(G)$. In Lemma~\ref{lemma:discard}, we find an upper-bound for $F_{\lambda'}(G)$ for some specific choices of 
two-terminal graphs $G$ in $T_{n,m}^{(1)}$. The set $T_{n,m}^{(\lambda')}$ is then determined in Lemma~\ref{lemma:maxlambda}. 

\begin{lemma}\label{lemma:discard}
Let $G$ be in $T_{n,m}^{(1)}$. If either the skeleton $G'$ is not isomorphic to $B_{n',m'}$, 
or $G'$ is isomorphic to $B_{n',m'}$ but the deletion of each minimum separator in $B_{n',m'}$ 
does not define a split subgraph of $G$, then $F_{\lambda'}(G) \leq b(n,m)\binom{m'}{\lambda'-1}$.
\end{lemma}
\begin{proof}
Let $G$ be any two-terminal graph in $T_{n,m}^{(1)}$ satisfying the conditions of the statement. We know that $b(G)=b(n,m)$, and each bridge of $G$ is an $s$-$t$ cut. Additionally, by Lemma~\ref{lemma:intermediate}\ref{i1}, the skeleton $G'$ belongs to $\mathcal{B}_0$. 
If $G'$ is not isomorphic to $B_{n',m'}$ then, by Proposition~\ref{proposition:lambda}, 
$\lambda(G')>\lambda'$. If $G'$ is isomorphic to $B_{n',m'}$ but the deletion of each minimum separator in $B_{n',m'}$ does not define a split subgraph of $G$, then the deletion of $\lambda'$ edges from $G'$ does not define a split subgraph of $G$. In either case, the value $F_{\lambda'}(G)$ equals the ways to delete $1$ out of $b(n,m)$ bridges times the ways to choose $\lambda'-1$ edges from $G'$. 
Consequently, $F_{\lambda'}(G) \leq b(n,m)\binom{m'}{\lambda'-1}$, as required.
\end{proof}

Now, let us consider any two-terminal graph $G$ in $T_{n,m}^{(1)}$ such that its skeleton $G'$ is isomorphic to the balloon graph $B_{n',m'}$. Proposition~\ref{proposition:lambda} gives that $\lambda(G')=\lambda'$. 
Denote $n(G')$ the number of minimum separators in $G'$ whose deletion define split subgraphs in $G$. Clearly, 
$n(G')\leq s(G')$. In order to count $F_{\lambda'}(G)$, observe that we can either remove $0$ or $1$ out of the $b(n,m)$ bridges and the remaining edges in $G'$ thus   
\begin{equation}\label{eq:terms}
F_{\lambda'}(G)=b(n,m)\binom{m'}{\lambda'-1}+n(G').
\end{equation}

If $d_G(s,t)=b(n,m)$ then the shortest path between $s$ and $t$ only consists of bridges and has no edges in $G'$ 
thus $n(G')=0$. For convenience we call $s'$ and $t'$ the vertices in $G'$ such that the distances $d_G(s,s')$ and $d_G(t,t')$ are minimum. 

\begin{definition}\label{definition:projected}
Let $G$ be any two-terminal graph in $T_{n,m}^{(1)}$ with terminals $s$ and $t$ whose skeleton is $B_{n',m'}$. Consider $s' = \arg\min_{y\in V(G')}\{d_G(s,y)\}$ and $t'= \arg\min_{y\in V(G')}\{d_G(t,y)\}$. If $s'\neq t'$ then 
the vertices $s'$ and $t'$ are called the \emph{projected terminals in $G'$}.   
\end{definition}
In the conditions of Definition~\ref{definition:projected}, the skeleton $G'$ of $G$ equipped with its projected terminals is a two-terminal graph in $T_{n',m'}$. 
Observe that if $G$ is in $T_{n,m}^{(1)}$ and $s'=t'$ then $d_G(s,t)=b(n,m)$, and 
$n(G')=0$. As we want to maximize the right hand side of equation~\eqref{eq:terms} we can assume from now on that $s'\neq t'$. 

\begin{lemma}\label{lemma:maxlambda}
Let $G$ be any two-terminal graph in $T_{n,m}^{(1)}$ satisfying the conditions of Definition~\ref{definition:projected}. Let $s'$ and $t'$ be the projected terminals in $G'$. All the following assertions hold.
\begin{enumerate}[label=(\roman*)]
\item\label{a1} If $\lambda'=n'-1$ then $F_{\lambda'}(G)$ does not depend on the choice of $G$.
\item\label{a2} If $\lambda'=n'-2$ then $F_{\lambda'}(G)$ is maximum if and only if $s't' \notin G$.
\item\label{a3} If $\lambda'<n'-2$ and $v'$ is the only vertex with minimum degree in the subgraph $G'$ of $G$ then $F_{\lambda'}(G)$ is maximum if and only if either $s'=v'$ or $t'=v'$. 
\end{enumerate}
\end{lemma}
\begin{proof}  
Let $G$ be any graph in $T_{n,m}^{(1)}$ such that $G' \cong B_{n',m'}$. Proposition~\ref{proposition:lambda} gives that $\lambda(G')=\lambda'$. Consider $G'$ equipped with the projected terminals $s'$ and $t'$. 
Denote $n(G')$ the number of split subgraphs of $G$ obtained by 
the deletion of precisely $\lambda'$ edges in $G'$. Observe that $n(G')$ is precisely $F_{\lambda'}(G')$. 
By equation~\eqref{eq:terms}, the maximization of $F_{\lambda'}(G)$ is equivalent to the maximization of $F_{\lambda'}(G')$ in the two-terminal graph $G'$. By Remark~\ref{remark:minimum}, there are precisely as many 
minimum separators $s(G')$ in $G'$ as vertices with minimum degree in $G'$. These minimum separators in $G'$ are precisely the incident edges of a fixed vertex with minimum degree in $G'$. Each split subgraph in $G'$ is obtained by the deletion of some of those minimum separators thus $F_{\lambda'}(G')\leq s(G')$. 
Additionally, only the deletion of all incident edges to $s'$ or to $t'$ define split subgraphs, and consequently, 
$F_{\lambda'}(G')\leq 2$. Thus far, we know that $n(G')=F_{\lambda'}(G') \leq \min\{s(G'),2\}$. 
Now, let us prove each of the statements separately.
\begin{enumerate}[label=(\roman*)]
\item If $\lambda'=n'-1$ then $G'$ is the complete graph $K_{n'}$ equipped with 
two terminals $s'$ and $t'$. Clearly, we obtain $2$ different split subgraphs by the deletion of all edges that are incident either to $s'$ or to $t'$ and $n(G')=2$. Consequently, the value 
$F_{\lambda'}(G)$ does not depend on the choice of $G$. 
\item If $\lambda'=n'-2$ then $G'$ is isomorphic to $K_{n'}-e$ for some edge $e$. 
By Remark~\ref{remark:minimum}, there are precisely $2$ minimum separators in $G$. 
These separators are obtained when we remove all edges in $G'$ that are incident to some fixed endpoint of $e$. Consequently, $n(G')\leq 2$, and the equality is obtained if and only if $s'$ and $t'$ are the endpoints of $e$. Equivalently, $F_{\lambda'}(G)$ is maximum if and only if $s't' \notin G$.
\item If $\lambda'<n'-2$ then, by Remark~\ref{remark:minimum}, the only minimum separator 
in $G'$ consists of $\lambda'$ edges in $G'$ that isolate the vertex $v'$ with minimum degree in $G'$. Consequently, $n(G')\leq 1$, and $n(G')=1$ if and only if the vertex with minimum degree in $G'$ is either $s'$ or $t'$. The lemma follows.
\end{enumerate}
\end{proof}

Observe that Lemma~\ref{lemma:maxlambda} characterizes the set 
$T_{n,m}^{(\lambda')}$. Now, let us pick an arbitrary two-terminal graph $G$ in $T_{n,m}^{(\lambda')}$. 
We know that $G'$ is isomorphic to $B_{n',m'}$ and the projected terminals $s'$ and $t'$ in $B_{n',m'}$ are different. Consider the two-terminal $G'$ equipped with its projected terminals. 
If only one out of $b(n,m)$ bridges in $G$ fails then its split reliability equals the reliability polynomial 
$R_{B_{n',m'}}(p)$; if no bridge in $G$ fails then its split reliability equals the split reliability of $G'$. Consequently, for all $p\in [0,1]$, 
\begin{equation} \label{eq:splitreliability}
SR_G(p) = b(n,m)(1-p)p^{b(n,m)-1}R_{G'}(p)+p^{b(n,m)}SR_{G'}(p).
\end{equation}

As a consequence we obtain the following result.
\begin{lemma}\label{lemma:firstcases}
Let $G$ be any two-terminal graph in $T_{n,m}^{\lambda'}$. If $\lambda' \in \{n'-1,n'-2\}$ then 
$\mathcal{G}_{n,m}=[G_{n,m}]$.
\end{lemma}
\begin{proof}
Let $H$ be any two-terminal graph in $T_{n,m}^{\lambda'}$. On the one hand, if $\lambda'=n'-1$ then its skeleton $H'$ is $K_{n'}$. On the other hand, if $\lambda'=n'-2$ then its skeleton $H'$ is $K_{n'}$ minus one edge $e$, and by Lemma~\ref{lemma:maxlambda}\ref{a2} we also know that the endpoints of  $e$ are 
precisely the projected terminals $s'$ and $t'$. Observe that, in either case, the skeleton $H'$ equipped with its projected terminals is isomorphic to the two-terminal balloon graph $G_{n',m'}$. 
By equation~\eqref{eq:splitreliability}, 
\begin{equation*}
SR_H(p) = b(n,m)(1-p)p^{b(n,m)-1}R_{B_{n',m'}}(p)+p^{b(n,m)}SR_{G_{n',m'}}(p) = SR_{G_{n,m}}(p).
\end{equation*}
Thus far, we obtained that the split reliability of each two-terminal graph $H$ in $T_{n,m}^{\lambda'}$ is precisely the same as the split reliability of the two-terminal balloon graph. By Remark~\ref{remark:splitequivalent}, any two-terminal graph $H$ in $T_{n,m}^{\lambda'}$ is split-equivalent to 
$G_{n,m}$, i.e., $T_{n,m}^{\lambda'}=[G_{n,m}]$. As each pair of two-terminal graphs in $T_{n,m}^{\lambda'}$ is split-equivalent, by Remark~\ref{remark:stop} we get that $T_{n,m}^{\lambda'}=\mathcal{G}_{n,m}$. Then, $\mathcal{G}_{n,m}=[G_{n,m}]$, and the lemma follows. 
\end{proof}

Lemma~\ref{lemma:maxn} studies a subsequent maximization of the numbers $F_{i}(G)$ when $\lambda' \leq n'-2$.
\begin{lemma}\label{lemma:maxn}
Let $G$ be any two-terminal graph in $T_{n,m}^{\lambda'}$, where $\lambda'\leq n'-2$. 
For each integer $i$ in $\{\lambda',\lambda'+1,\ldots,n'-3\}$ the number $F_{i}(G)$ does not depend on the choice of $G$. Additionally, $F_{n'-2}(G)$ is maximum if and only if $s't' \notin G$.
\end{lemma}
\begin{proof}
Let $G$ be any two-terminal graph in $T_{n,m}^{\lambda'}$, where $\lambda'\leq n'-2$. We already know that 
its skeleton $G'$ is isomorphic to $B_{n',m'}$ and is equipped with 
two different terminals $s'$ and $t'$. By Lemma~\ref{lemma:maxlambda}\ref{a3} we can assume without loss of generality that $s'$ is the only vertex having degree $\lambda'$ in $G'$, i.e., $s'=v'$. 
As $\lambda'\leq n'-2$, by Remark~\ref{remark:minimum} we know that $G'$ has only one minimum separator. Let $U$ be the only minimum separator in $G'$ consisting of all edges in $G'$ that are incident to $s'$. Let $i \in \{\lambda',\lambda'+1,\ldots,n'-3\}$. 
If we remove $0$ bridge then a split subgraph of $G$ is obtained when we remove all edges in $U$ plus $i-\lambda'$ edges in $G'-U$. If we remove $1$ bridge then a split subgraph of $G$ is only obtained when we remove possibly some (but not all) edges in $U$ and the remaining edges in $G'-U$. Consequently,
\begin{equation}\label{eq:i}
F_{i}(G) = \binom{m'-\lambda'}{i-\lambda'}+
b(n,m)\sum_{j=0}^{\lambda'-1}\binom{\lambda'}{j}\binom{m'-\lambda'}{i-j-1},   
\end{equation}
and $F_{i}(G)$ does not depend on the choice of $G$. 

To prove the second part of the statement consider $i=n'-2$. 
On the one hand, if $s't'\in G'$ then the degree of $t'$ in $G'$ is $n'-1$. The previous reasoning holds to find all split subgraphs in $G$ and $F_{n'-2}(G)$ is ruled by Equation~\eqref{eq:i} letting $i=n'-2$. On the other hand, if $s't' \notin G'$ then the degree of $t'$ in $G'$ is $n'-2$. An additional split subgraph is obtained by removing precisely all $n'-2$ edges incident to $t'$. Therefore, 
$F_{n'-2}(G)$ is maximum if and only if $s't' \notin G$, and the lemma follows.
\end{proof}

We are in conditions to prove the main result of this section.\\

\emph{Proof of Theorem~\ref{theorem:local1}:} 
Let $(n,m)$ an arbitrary pair of integers in $I$. Define the numbers $n'$, $m'$ and $\lambda'$ by $n'=n-b(n,m)$, $m'=m-b(n,m)$, and $\lambda'=\lambda(n',m')$, respectively. If $\lambda' \in \{n'-1,n'-2\}$ 
then Lemma~\ref{lemma:firstcases} gives that $\mathcal{G}_{n,m}=[G_{n,m}]$. 
If otherwise, by Lemma~\ref{lemma:maxn} we get that $G$ is in $T_{n,m}^{n'-2}$ if and only if  
the skeleton $G'$ is isomorphic to $B_{n',m'}$ and its projected terminals $s'$ and $t'$ are nonadjacent. 
As the two-terminal balloon graph $G_{n,m}$ satisfies the previous conditions, $G_{n,m}$ belongs to $T_{n,m}^{n'-2}$. A direct computation of the split reliability of $G$ gives that the polynomial $SR_{G}(p)$ 
is given in equation~\eqref{eq:splitreliability}, regardless of the choice of the two-terminal graph $G$ in $T_{n,m}^{n'-2}$. By Remark~\ref{remark:splitequivalent}, each pair of two-terminal graphs in $T_{n,m}^{n'-2}$ are split-equivalent. By Remark~\ref{remark:splitequivalent}, $T_{n,m}^{n'-2}=\mathcal{G}_{n,m}$. 
Finally, as $G_{n,m}$ belongs to $T_{n,m}^{n'-2}$, we conclude that $T_{n,m}^{n'-2}=[G_{n,m}]$. 
Consequently, $\mathcal{G}_{n,m}=[G_{n,m}]$, and the theorem follows. \qed

 \section{Nonexistence of uniformly most split reliable graphs}\label{section:main}
In this section we will construct, for each pair of integers $n$ and $m$ such that $n\geq 7$ and $n \leq m\leq \binom{n-3}{2}+3$, a two-terminal graph $H_{n,m}$ in $T_{n,m}$ such that $N_{n-2}(H_{n,m})>N_{n-2}(G_{n,m})$. We will choose $H_{n,m}$ as one of the graphs in $\{H_{n,m}^{(0)},H_{n,m}^{(1)},H_{n,m}^{(2)}\}$, where the definitions of $H_{n,m}^{(0)}$, $H_{n,m}^{(1)}$ and $H_{n,m}^{(2)}$ are given in the following paragraphs. 
Throughout this section $s'$ is the vertex with minimum degree in $G_{n,m}'$ and $\lambda'$ is the degree of $s'$ in the subgraph $G_{n,m}'$ of $G_{n,m}$, i.e., $d_{G_{n,m}'}(s')=\lambda'$. 

\begin{definition}\label{def:H0}
Let $e_b$ and $e_0$ be two edges in $G_{n,m}$ such that $e_b$ is a bridge and $e_0$ is an edge in $G_{n,m}'$ where $s'$ is one of its endpoints. The two-terminal graph $H_{n,m}^{(0)}$ is obtained from $G$ by the contraction of $e_b$ followed by a subdivision of $e_0$. 
\end{definition}

\begin{definition}\label{def:H1}
Assume that there exists two vertices $y$ and $z$ in $G_{n,m}'$ that are not adjacent to $s'$. Let $e_1$ be $yz$ and $e_b$ be some bridge in $G_{n,m}$. The two-terminal graph $H_{n,m}^{(1)}$ is obtained from $G$ by the contraction of $e_b$ followed by a subdivision of $e_1$. 
\end{definition}

\begin{definition}\label{def:H2}
Let $e_b$ and $e_2$ be two edges in $G_{n,m}$ such that $e_b$ is a bridge, $e_2$ is in $G_{n,m}'$, and the endpoints of $e_2$ are both adjacent to $s'$. The two-terminal graph $H_{n,m}^{(2)}$ is obtained from $G$ by the contraction of $e_b$ followed by a subdivision of $e_2$. 
\end{definition}

\begin{figure}[ht]
\begin{center}
\scalebox{1.3}{
\begin{tabular}{c}
\begin{tikzpicture}[scale=1.2]
\coordinate(d) at (0.5,1.5);
\coordinate(c2) at (1,1);
\coordinate(c1) at (0,1); 
\coordinate(b2) at (1,0.5);
\coordinate(b1) at (0,0.5); 
\coordinate(a) at (0.5,0);
\coordinate(e) at (0.5,-0.5);
\coordinate(f) at (0.5,-1);
\coordinate(g) at (0.5,-1.5);

\draw (a)--(b1);
\draw (a)--(b2);
\draw (b1)--(b2);
\draw (b1)--(c1);
\draw (b1)--(c2);
\draw (b1)--(d);
\draw (b2)--(c1);
\draw (b2)--(c2);
\draw (b2)--(d);
\draw (c1)--(c2);
\draw (c1)--(d);
\draw (c2)--(d);
\draw (a)--(e);
\draw (e)--(f);
\draw (f)--(g);

\foreach \p in {a,b1,b2,c1,c2,d,e,f,g}
       \fill [black] (\p) circle (2.2pt);
\end{tikzpicture}
\\
\vspace{-1mm}
$B_{9,15}$
\end{tabular}
\begin{tabular}{c}
\begin{tikzpicture}[scale=1.2]
\coordinate[label=below:$t$](d) at (0.5,1.5);
\coordinate(c2) at (1,1);
\coordinate(c1) at (0,1); 
\coordinate(b2) at (1,0.5);
\coordinate(b1) at (0,0.5); 
\coordinate[label=left:$s'$](a) at (0.5,0);
\coordinate(e) at (0.5,-0.5);
\coordinate(f) at (0.5,-1);
\coordinate[label=left:$s$](g) at (0.5,-1.5);

\draw (a)--(b1);
\draw (a)--(b2);
\draw (b1)--(b2);
\draw (b1)--(c1);
\draw (b1)--(c2);
\draw (b1)--(d);
\draw (b2)--(c1);
\draw (b2)--(c2);
\draw (b2)--(d);
\draw (c1)--(c2);
\draw (c1)--(d);
\draw (c2)--(d);
\draw (a)--(e);
\draw (e)--(f);
\draw (f)--(g);

\foreach \p in {a,b1,b2,c1,c2,d,e,f,g}
       \fill [black] (\p) circle (2.2pt);
\end{tikzpicture}
\\
\vspace{-1mm}
$G_{9,15}$
\end{tabular}
\begin{tabular}{c}
\begin{tikzpicture}[scale=1.2]
\coordinate[label=below:$t$](d) at (0.5,1.5);
\coordinate(c2) at (1,1);
\coordinate(c1) at (0,1); 
\coordinate(b2) at (1,0.5);
\coordinate(b1) at (0,0.5); 
\coordinate(x) at (0.25,0.25); 
\coordinate[label=left:$s'$](a) at (0.5,0);
\coordinate(e) at (0.5,-0.5);
\coordinate[label=left:$s$](f) at (0.5,-1);
\coordinate(g) at (0.5,-1.5);

\draw (a)--(x);
\draw (x)--(b1);
\draw (a)--(b2);
\draw (b1)--(b2);
\draw (b1)--(c1);
\draw (b1)--(c2);
\draw (b1)--(d);
\draw (b2)--(c1);
\draw (b2)--(c2);
\draw (b2)--(d);
\draw (c1)--(c2);
\draw (c1)--(d);
\draw (c2)--(d);
\draw (a)--(e);
\draw (e)--(f);

\foreach \p in {a,b1,b2,c1,c2,d,e,f,x}
       \fill [black] (\p) circle (2.2pt);
\foreach \p in {g}
       \fill [white] (\p) circle (2.2pt);
\end{tikzpicture}
\\
\vspace{-1mm}
$H_{9,15}^{(0)}$
\end{tabular}
\begin{tabular}{c}
\begin{tikzpicture}[scale=1.2]
\coordinate[label=below:$t$](d) at (0.5,1.5);
\coordinate(c2) at (1,1);
\coordinate(x) at (0.5,1);
\coordinate(c1) at (0,1); 
\coordinate(b2) at (1,0.5);
\coordinate(b1) at (0,0.5); 
\coordinate[label=left:$s'$](a) at (0.5,0);
\coordinate(e) at (0.5,-0.5);
\coordinate[label=left:$s$](f) at (0.5,-1);
\coordinate(g) at (0.5,-1.5);

\draw (a)--(b1);
\draw (a)--(b2);
\draw (b1)--(b2);
\draw (b1)--(c1);
\draw (b1)--(c2);
\draw (b1)--(d);
\draw (b2)--(c1);
\draw (b2)--(c2);
\draw (b2)--(d);
\draw (c1)--(x);
\draw (x)--(c2);
\draw (c1)--(d);
\draw (c2)--(d);
\draw (a)--(e);
\draw (e)--(f);

\foreach \p in {a,b1,b2,c1,c2,d,e,f,x}
       \fill [black] (\p) circle (2.2pt);
\foreach \p in {g}
       \fill [white] (\p) circle (2.2pt);
\end{tikzpicture}
\\
\vspace{-1mm}
$H_{9,15}^{(1)}$
\end{tabular}
\begin{tabular}{c}
\begin{tikzpicture}[scale=1.2]
\coordinate[label=below:$t$](d) at (0.5,1.5);
\coordinate(c2) at (1,1);
\coordinate(c1) at (0,1); 
\coordinate(b2) at (1,0.5);
\coordinate(x) at (0.5,0.5);
\coordinate(b1) at (0,0.5); 
\coordinate[label=left:$s'$](a) at (0.5,0);
\coordinate(e) at (0.5,-0.5);
\coordinate[label=left:$s$](f) at (0.5,-1);
\coordinate(g) at (0.5,-1.5);

\draw (a)--(b1);
\draw (a)--(b2);
\draw (b1)--(b2);
\draw (b1)--(c1);
\draw (b1)--(c2);
\draw (b1)--(d);
\draw (b2)--(c1);
\draw (b2)--(c2);
\draw (b2)--(d);
\draw (c1)--(c2);
\draw (c1)--(d);
\draw (c2)--(d);
\draw (a)--(e);
\draw (e)--(f);

\foreach \p in {a,b1,b2,c1,c2,d,e,f,x}
       \fill [black] (\p) circle (2.2pt);
\foreach \p in {g}
       \fill [white] (\p) circle (2.2pt);
\end{tikzpicture}
\\
\vspace{-1mm}
$H_{9,15}^{(2)}$
\end{tabular}
}
\end{center}
\caption{Graph $B_{9,15}$ and two-terminal graphs 
$G_{9,15}$, $H_{9,15}^{(0)}$,  $H_{9,15}^{(1)}$ and $H_{9,15}^{(2)}$.\label{fig:GyH}}
\end{figure}
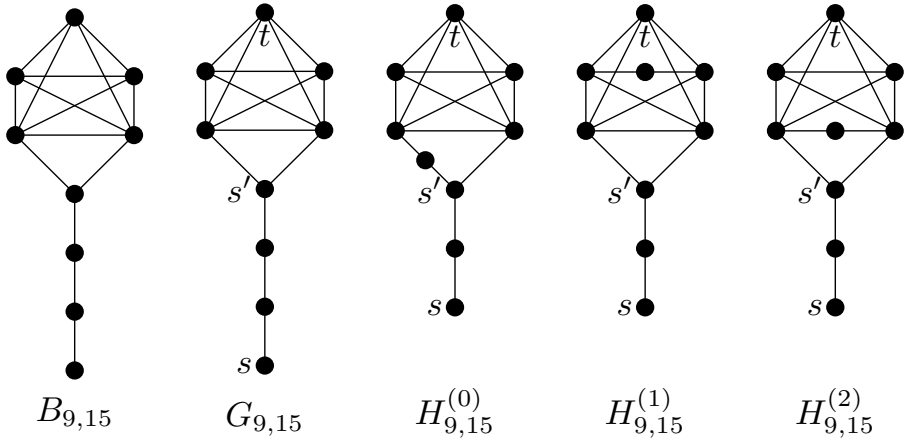

The reader can observe that, for each $i\in \{0,1,2\}$, 
the resulting two-terminal graph $H_{n,m}^{(i)}$ obtained by the contraction of a bridge $e_b$ followed by a subdivision of $e_i$ does not depend on the choices of the edges $e_b$ and $e_i$. Figure~\ref{fig:GyH} depicts the graph $B_{9,15}$ as well as the two-terminal graphs $G_{9,15}$, $H_{9,15}^{(0)}$, $H_{9,15}^{(1)}$ and $H_{9,15}^{(2)}$. 

For each $G$ in $T_{n,m}$ we denote $t_2(G)$ the number of split subgraphs of $G$ consisting of $2$ disjoint trees. The following remarks will be useful. 
\begin{remark}\label{remark:positive}
Let $(n,m)$ be in $I_1$ and $G$ be $G_{n,m}$. We will show that for any edge $e$ in $G'$ it holds that $t_2(G'-e)>0$. 
In fact, as $m\geq n$ we know that $G'$ has some cycle. Then, $G'-e$ has at least one spanning tree $T$. 
Pick some edge $e'$ in $T$ included in the only path between $s'$ and $t'$ in $T$. The split subgraph $T-e$ consists precisely of $2$ disjoint trees, and $t_2(G'-e)>0$.   
\end{remark}

\begin{remark}\label{remark:3bridges}
For each $(n,m)$ in $I$ such that $n\leq m\leq \binom{n-3}{2}+3$ it holds that $b(n,m)\geq 3$. In fact, by Proposition~\ref{proposition:balloon}, $b(n,\binom{n-3}{2}+3)=
b(B_{n,\binom{n-3}{2}+3})=3$. As the addition of edges in a connected graph does not create bridges, the sequence $b(n,m)$ is decreasing with respect to $m$ thus $b(n,m)\geq b(n,\binom{n-3}{2}+3)=3$. 
\end{remark}

\begin{lemma}\label{lemma:edge}
Let $(n,m)$ be in $I_1$. Let $G$ be $G_{n,m}$ and $H$ obtained from $G$ by a contraction of some bridge of $G$ followed by a subdivision of some edge $e$ in $G'$. Then,
\begin{equation}
N_{n-2}(H)-N_{n-2}(G)>(b(n,m)-1)t(G'-e)-t(G').    
\end{equation}
\end{lemma}
\begin{proof}
As $(n,m)$ is in $I_1$, by Proposition~\ref{proposition:balloon} we get that  $b(n,m)=b(B_{n,m})\geq 1$. By construction, $H$ has precisely $b(n,m)-1$ bridges. 
Now, let us count $N_{n-2}(G)$ and $N_{n-2}(H)$. 
If we remove $2$ or more bridges in $G$ or in $H$ then we will obtain more than $2$ connected components thus no split subgraphs are produced in this way. 
On the one hand, if we remove $0$ bridge then we must count split subgraph consisting of $2$ disjoint trees in the respective skeletons 
$G'$ or $H'$. On the other hand, for each choice of $1$ bridge in $G$ or $H$ there are as many split subgraphs in $G$ or $H$ as spanning trees in $G'$ or $H'$. Then, 
\begin{align}
N_{n-2}(G) &= b(n,m)t(G')+t_2(G'), \label{eqG}\\
N_{n-2}(H) &= (b(n,m)-1)t(H')+t_2(H'), \label{eqH}
\end{align}    
Let $e_1$ and $e_2$ be the resulting edges in $H$ after the subdivision of $e$ in $G$. Observe that $t(H'-e_1)=t(G'-e)$ and $t(H'*e_1)=t(G')$. Then,
\begin{equation}
t(H') = t(H'-e_1)+t(H'*e_1)=t(G'-e)+t(G'). \label{eqH1}
\end{equation}
Analogously,
\begin{equation}\label{eqH2}
t_2(H') = t_2(G'-e)+t_2(G').    
\end{equation}
Replacing Equations~\eqref{eqH1}~and~\eqref{eqH2} into \eqref{eqH},
\begin{equation}\label{eqH3}
N_{n-2}(H) = (b(n,m)-1)(t(G'-e)+t(G'))+t_2(G'-e)+t_2(G'). 
\end{equation}
Now, we can take the difference between Equations~\eqref{eqH3} and~\eqref{eqG} and use Remark~\ref{remark:positive} 
to get that
\begin{equation*}
N_{n-2}(H)-N_{n-2}(G)=(b(n,m)-1)t(G'-e)-t(G')+t_2(G'-e)>(b(n,m)-1)t(G'-e)-t(G'). \qedhere    
\end{equation*}
\end{proof}

\begin{lemma}\label{lemma:1}
Let $n$ and $m$ be integers such that $n\geq 7$ and $n \leq m \leq \binom{n-3}{2}+3$ and let $G$ be $G_{n,m}$. 
If $G'$ has minimum degree $\lambda'$ such that $\lambda'\geq 3$ then $N_{n-2}(H_{n,m}^{(0)})>N_{n-2}(G)$.
\end{lemma}
\begin{proof}
By our assumption, $n\geq 7$ and $n \leq m \leq \binom{n-3}{2}+3$. Then, Remark~\ref{remark:3bridges} gives that $b(n,m)\geq 3$. Let $s'$ be the vertex with minimum degree $\lambda'$ in $G'$, and let $n'$ be the number of vertices in $G'$. By our assumption the edge connectivity of $G'$ is at least $3$ thus $G'$ has at least $4$ vertices, i.e., $n'\geq 4$. 
Let us pick any edge $e$ in $G'$ such that one of its endpoints is $s'$ 
and one bridge $e_b$ in $G$. Construct the two-terminal graph $H_{n,m}^{(0)}$ 
from $G$ by the contraction of $e_b$ followed by a subdivision of $e$. By Lemma~\ref{lemma:edge}, 
\begin{equation}\label{eq:partial1}
N_{n-2}(H_{n,m}^{(0)})-N_{n-2}(G)>(b(n,m)-1)t(G'-e)-t(G').    
\end{equation}
Now, we will use Bogdanowicz formula given in Equation~\eqref{eq:bogdanowicz} to find $t(G')$ and $t(G'-e)$. On the one hand, $G'$ is the graph $K_{n'-1}$ plus one vertex $s'$ adjacent to precisely $\lambda'$ vertices of $K_{n'-1}$. Therefore, $G' \cong H(n';\lambda')$. On the other hand, $G'-e \cong H(n';\lambda'-1)$. Applying Equation~\eqref{eq:bogdanowicz} to both $G'$ and $G'-e$,
\begin{align*}
t(G')&= t(H(n';\lambda'))=\lambda'(n')^{\lambda'-1}(n'-1)^{n'-\lambda'-2};\\
t(G'-e)&= t(H(n';\lambda'-1))=(\lambda'-1)(n')^{\lambda'-2}(n'-1)^{n'-\lambda'-1}.
\end{align*}
Consequently, 
\begin{equation}\label{eq:half1}
t(G'-e) = \frac{\lambda'-1}{\lambda'}\frac{n'-1}{n'}t(G')\geq \left(\frac{2}{3}\right)\left(\frac{3}{4}\right)t(G')=\frac{t(G')}{2},
\end{equation}
where we used that $\lambda'\geq 3$ and $n'\geq 4$. Replacing~\eqref{eq:half1} into~\eqref{eq:partial1} and then using that $b(n,m)\geq 3$ gives 
\begin{equation*}
N_{n-2}(H_{n,m}^{(0)})-N_{n-2}(G)>(b(n,m)-1)\frac{t(G')}{2}-t(G')\geq 0,
\end{equation*}
thus $N_{n-2}(H_{n,m}^{(0)})>N_{n-2}(G)$, and the lemma follows.
\end{proof}

\begin{lemma}\label{lemma:2}
Let $n$ and $m$ be integers such that $n\geq 7$ and $n \leq m \leq \binom{n-3}{2}+3$ and let $G$ be $G_{n,m}$. 
If $G'$ has $n'$ vertices and minimum degree $\lambda'$ such that $n'\geq 5$ and $\lambda'\leq n'-3$,  
then $N_{n-2}(H_{n,m}^{(1)})>N_{n-2}(G)$.    
\end{lemma}
\begin{proof}
By our assumption, $n\geq 7$ and $n \leq m \leq \binom{n-3}{2}+3$. Then, Remark~\ref{remark:3bridges} gives that $b(n,m)\geq 3$. 
Let $s'$ be the vertex with minimum degree $\lambda'$ in $G'$. As $\lambda'\leq n'-3$, there exists at least two vertices $y$ and $z$ in $G'$ that are not adjacent to $s'$. Let $e_b$ be some bridge in $G$. Construct the two-terminal graph $H_{n,m}^{(1)}$ from $G$ by the contraction of $e_b$ followed by a subdivision of the edge $e=yz$. 
By Lemma~\ref{lemma:edge}, 
\begin{equation}\label{eq:partial2}
N_{n-2}(H_{n,m}^{(1)})-N_{n-2}(G)>(b(n,m)-1)t(G'-e)-t(G').    
\end{equation}
Now, we will use Bogdanowicz formula given in Equation~\eqref{eq:bogdanowicz} to find $t(G')$ and $t(G'-e)$. On the one hand, $G'\cong H(n';\lambda')$. 
On the other hand, $G'-e'$ consists of a complete graph $K_{n'-2}$ plus two 
independent vertices $\{s',y\}$ such that $s'$ is adjacent to $\lambda'$ vertices and $y$ is adjacent to all vertices of $K_{n'-2}$ except for $z$. Therefore, $G'-e \cong H(n';n'-3,\lambda')$. 
Applying Equation~\eqref{eq:bogdanowicz} to both $G'$ and $G'-e$,
\begin{align*}
t(G')&= t(H(n';\lambda'))=\lambda'(n')^{\lambda'-1}(n'-1)^{n'-\lambda'-2};\\
t(G'-e)&= t(H(n';n'-3,\lambda')) = \lambda'(n')^{\lambda'-1}(n'-3)(n'-1)^{n'-\lambda'-3}.
\end{align*}
Consequently, 
\begin{equation}\label{eq:half2}
t(G'-e) = \frac{n'-3}{n'-1}t(G')\geq \frac{t(G')}{2},
\end{equation}
where we used in the inequality that $n'\geq 5$. Replacing~\eqref{eq:half2} into~\eqref{eq:partial2} and using that $b(n,m)\geq 3$ gives 
\begin{equation*}
N_{n-2}(H_{n,m}^{(1)})-N_{n-2}(G)>(b(n,m)-1)\frac{t(G')}{2}-t(G')\geq 0,
\end{equation*}
thus $N_{n-2}(H_{n,m}^{(1)})>N_{n-2}(G)$, and the lemma follows.
\end{proof}

\begin{lemma}\label{lemma:3}
Let $n$ and $m$ be integers such that $n\geq 7$ and $n \leq m \leq \binom{n-3}{2}+3$ and let $G$ be $G_{n,m}$. 
If $G'$ has $n'$ vertices and minimum degree $\lambda'$ 
such that $n'=4$ and $\lambda'=2$ then $N_{n-2}(H_{n,m}^{(2)})>N_{n-2}(G)$.    
\end{lemma}
\begin{proof}
By our assumption, $n\geq 7$ and $n \leq m \leq \binom{n-3}{2}+3$. Then, Remark~\ref{remark:3bridges} gives that $b(n,m)\geq 3$. In the conditions of the statement $G' \cong K_4-e$ and its terminals are nonadjacent. Let 
$\{y,z,s',t'\}$ the vertices in $G'$ and let $e_b$ be any bridge of $G$. 
Construct the two-terminal graph $H_{n,m}^{(2)}$ from $G$ by the contraction of $e_b$ followed by a subdivision of the edge $e=yz$. 
A simple counting shows that $t(G')=4$ and $t_2(G')=8$, while 
$t(H')=8$ and $t_2(H')=12$. Consequently,
\begin{align*}
N_{n-2}(G) &= b(n,m)t(G')+t_2(G')=4b(n,m)+8;\\
N_{n-2}(H)&=(b(n,m)-1)t(H')+t_2(H')=8(b(n,m)-1)+12=8b(n,m)+4.
\end{align*}
As $b(n,m)\geq 3$ we get that $N_{n-2}(H)=4b(n,m)+(4b(n,m)+4)>4b(n,m)+8=N_{n-2}(G)$, and the lemma follows.
\end{proof}

We are in position to prove Proposition~\ref{proposition:near0}.\\

\emph{Proof of Proposition~\ref{proposition:near0}:} 
Let $n$ and $m$ be two integers as in the statement. 
As $B_{n',m'}$ has no bridges, $\lambda'\geq 2$. 
If $\lambda'>2$ then, by Lemma~\ref{lemma:1}, the graph $H$ given by $H_{n,m}^{(0)}$ satisfies that 
$N_{n-2}(H)>N_{n-2}(G_{n,m})$. If $\lambda'=2$ then, clearly, $n'\geq 3$. If $\lambda'=2$ and $n'\geq 5$ then, by Lemma~\ref{lemma:2}, the graph $H$ given by $H_{n,m}^{(1)}$ 
satisfies that $N_{n-2}(H)>N_{n-2}(G_{n,m})$. If $\lambda'=2$ and $n'=3$ then $n=m$ and it has been already proved by Brown and McMullin~\cite{2023-Brown} that there exists a two-terminal graph $H_{n,n}$ in $T_{n,n}$ such that $N_{n-2}(H_{n,n})>N_{n-2}(G_{n,n})$ when $n\geq 6$. Finally, if $\lambda'=2$ and $n'=4$ then, by Lemma~\ref{lemma:3}, the graph $H$ given by $H_{n,m}^{(2)}$ satisfies that $N_{n-2}(H)>N_{n-2}(G_{n,m})$. The proposition follows. \qed 
\\ \\

Once we proved Proposition~\ref{proposition:near0} we know from Section~\ref{section:strategy} that Theorem~\ref{theorem:nonexistence} also holds. Thus far, the decision of existence or nonexistence of uniformly most split reliable graphs in 
 each of the classes $T_{n,m}$ such that $n\geq 7$ and $\binom{n-3}{2}+4 \leq m \leq \binom{n}{2}-2$ is an open problem. Observe that if $G$ is uniformly most split reliable two-terminal in any of the previous nonempty classes $T_{n,m}$ then, by Theorem~\ref{theorem:local1}, $G$ is equivalent to $G_{n,m}$ and thus $G_{n,m}$ would also be uniformly most split reliable in $T_{n,m}$. Future work should be carried out to determine whether or not $G_{n,m}$ is uniformly most split reliable in such nonempty sets $T_{n,m}$. 

\section*{Acknowledgments}
This work is partially supported by City University of New York project entitled \emph{On the problem of characterizing graphs with maximum number of spanning trees}, grant number 66165-00. I thank Dr. Mart\'in Safe for his helpful comments that improved the presentation of this manuscript.

\end{document}